\documentclass[12pt,reqno]{amsart}

\usepackage{amssymb}
\usepackage[all]{xy}

\usepackage{hyperref}

\numberwithin{equation}{section}

\newtheorem{theorem}{Theorem}[section]
\newtheorem{proposition}[theorem]{Proposition}
\newtheorem{lemma}[theorem]{Lemma}
\newtheorem{corollary}[theorem]{Corollary}

\theoremstyle{definition}
\newtheorem{definition}[theorem]{Definition}
\newtheorem{remark}[theorem]{Remark}

\begin{document}

\baselineskip=15pt

\title[A criterion for holomorphic Lie algebroid connections]{A criterion for holomorphic Lie algebroid connections}

\author[D. Alfaya]{David Alfaya}

\address{Department of Applied Mathematics and Institute for Research in Technology, ICAI 
School of Engineering, Comillas Pontifical University, C/Alberto Aguilera 25, 28015 Madrid, 
Spain}

\email{dalfaya@comillas.edu}

\author[I. Biswas]{Indranil Biswas}

\address{Department of Mathematics, Shiv Nadar University, NH91, Tehsil Dadri,
Greater Noida, Uttar Pradesh 201314, India}

\email{indranil.biswas@snu.edu.in, indranil29@gmail.com}

\author[P. Kumar]{Pradip Kumar}

\address{Department of Mathematics, Shiv Nadar University, NH91, Tehsil Dadri,
Greater Noida, Uttar Pradesh 201314, India}

\email{Pradip.Kumar@snu.edu.in}

\author[A. Singh]{Anoop Singh}

\address{Department of Mathematical Sciences, Indian Institute of Technology (BHU), Varanasi 221005, India}

\email{anoopsingh.mat@iitbhu.ac.in}



\date{}

\begin{abstract}
Given a holomorphic Lie algebroid $(V,\, \phi)$ on a compact connected Riemann surface $X$, we give
a necessary and sufficient condition for a holomorphic vector bundle $E$ on $X$ to admit a holomorphic
Lie algebroid connection. If $(V,\, \phi)$ is nonsplit, then every holomorphic vector bundle on $X$ admits a
holomorphic Lie algebroid connection for $(V,\, \phi)$. If $(V,\, \phi)$ is split, then a holomorphic vector
bundle $E$ on $X$ admits a holomorphic Lie algebroid connection if and only if the degree of each
indecomposable component of $E$ is zero.
\end{abstract}

\maketitle

\tableofcontents

\section{Introduction}

Take a compact connected Riemann surface $X$. A holomorphic connection on a holomorphic vector bundle $E$
on $X$ is a holomorphic differential operator
$$
D\ :\ E\ \longrightarrow\ E\otimes (TX)^*
$$
such that $D(fs) \,=\, fD(s) + s\otimes df$ for all locally defined holomorphic sections $s$ of $E$
and all locally defined holomorphic functions $f$ on $X$. It is a well known fact that not every holomorphic
vector bundle admits a holomorphic connection. For any holomorphic vector bundle $E$ on $X$, consider
a holomorphic decomposition of $E$
$$
E \ =\ \bigoplus_{i=1}^\ell E_i
$$
into a direct sum of indecomposable holomorphic vector bundles. A theorem of Atiyah says that for any
other holomorphic decomposition of $E$ into a direct sum of indecomposable holomorphic vector bundles, the
isomorphism classes of the direct summands are simply a permutation of the isomorphism classes of
$E_i$, $1\, \leq\, i\, \leq\, \ell$, \cite{At1}. The holomorphic vector bundle $E$ admits a holomorphic
connection if and only if the degree of every $E_i$, $1\, \leq\, i\, \leq\, \ell$, is zero \cite{At2}, \cite{We}.

A Lie algebroid on $X$ is a holomorphic vector bundle $V\, \longrightarrow\, X$, together with an
${\mathcal O}_X$--linear homomorphism $\phi\, :\, V\, \longrightarrow\, TX$ called its anchor, and endowed with
a structure of a $\mathbb{C}$--Lie algebra on the sheaf of locally defined holomorphic sections of $V$
$$
[-, \,-] \ :\ V\otimes_{\mathbb C} V \ \longrightarrow\ V
$$
such that $[s,\, f\cdot t]\,=\, f\cdot [s,\, t]+\phi(s)(f)\cdot t$ for all locally defined
holomorphic sections $s,\, t$ of $V$ and all locally defined holomorphic functions $f$ on $X$.
Note that the pair $(TX,\, {\rm Id}_{TX})$ defines a Lie algebroid; the Lie algebra structure on
$TX$ is given by the Lie bracket operation of vector fields.

Given a Lie algebroid $(V,\, \phi)$ on $X$, a Lie algebroid connection on a holomorphic vector bundle
$E$ on $X$ for $(V,\, \phi)$ is a holomorphic differential operator
$$
D\ :\ E\ \longrightarrow\ E\otimes V^*
$$
such that $D(fs) \,=\, fD(s) + s\otimes \phi^*(df)$, where $\phi^*$ is the dual homomorphism
for $\phi$, and, as before, $s$ is any locally defined holomorphic section of $E$ and $f$ is any locally defined
holomorphic function on $X$. When $(V,\, \phi)\,=\, (TX,\, {\rm Id}_{TX})$ is the natural Lie algebroid structure on
the tangent bundle, a $(V,\, \phi)$--connection on $E$ is a holomorphic connection on $E$ in the usual sense.

It may be mentioned that by
choosing the Lie algebroid appropriately, a wide range of other algebraic and differential geometric objects
can be interpreted as Lie algebroid connections as well. For instance, Higgs bundles, \cite{Hi, Si1}, twisted Higgs
bundles, \cite{Ni1,GGPN}, logarithmic connections, \cite{De,Ni2}, meromorphic connections, \cite{Bo,BS}, and a broad
subclass of Simpson's notion of $\Lambda$-modules, \cite{Si2, To2}, can all be understood as Lie algebroid connections
for suitable choices of the Lie algebroid $(V,\, \phi)$.

Our aim here is to find a complete necessary and sufficient condition for a holomorphic vector bundle $E$ to admit
a Lie algebroid connection for a given Lie algebroid $(V,\, \phi)$.

A Lie algebroid $(V,\, \phi)$ on $X$ will be called split if there is an ${\mathcal O}_X$--linear homomorphism
$$
\gamma\, :\, TX \, \longrightarrow\, V
$$
such that $\phi\circ\gamma\,=\, {\rm Id}_{TX}$. A Lie algebroid $(V,\, \phi)$ on $X$ will
be called nonsplit if it is not split.

We prove the following (see Theorem \ref{thm1} and Proposition \ref{prop1}):

\begin{theorem}\label{thm-i}\mbox{}
\begin{itemize}
\item Let $(V,\, \phi)$ be a nonsplit Lie algebroid. Then every holomorphic vector bundle $E$ on $X$ admits a
Lie algebroid connection for $(V,\, \phi)$.

\item Let $(V,\, \phi)$ be a split Lie algebroid. Then the following two statements are equivalent:
\begin{enumerate}
\item $E$ admits a Lie algebroid connection for $(V,\, \phi)$.

\item Each indecomposable component of $E$ is of degree zero.
\end{enumerate}
\end{itemize}
\end{theorem}

Since any stable vector bundle is indecomposable, if $V$ is stable with ${\rm rank}(V)\, \geq\, 2$, then any Lie 
algebroid $(V,\, \phi)$ is automatically nonsplit. In \cite{BKS} it was proved that for any Lie algebroid 
$(V,\,\phi)$ where $V$ is stable with ${\rm rank}(V)\, \geq\, 2$, every holomorphic vector bundle $E$ admits a Lie 
algebroid connection for $(V,\,\phi)$. When $V$ is stable, some partial results on the existence of flat Lie 
algebroid connections on a stable vector bundle were also given earlier in \cite{AO}. On the other hand, in 
\cite{To1} \cite{To2} a notion of a generalized Atiyah class $a_V(E)\,\in\, H^1(X,\, \text{End}(E)\otimes V^*)$ for Lie 
algebroid connections was described by Tortella such that $a_V(E)\,=\,0$ if and only if the vector bundle $E$ admits
a Lie algebroid connection. The construction given here for an Atiyah exact sequence for 
Lie algebroids (see Section \ref{sec3}) differs from the one in \cite{To1} and \cite{To2}.

The paper is organized as follows. We study the problem for split and nonsplit Lie algebroids separately. Section 
\ref{sec2} introduces some general properties of Lie algebroid connections and describes a necessary and 
sufficient condition for the existence of such connections on an arbitrary vector bundle for split Lie algebroids. 
Section \ref{sec3} explores different conditions characterizing the existence of Lie algebroid connections on a 
vector bundle based on a generalization of Atiyah's exact sequence for Lie algebroids. Through these 
characterizations, in Section \ref{sec4} the existence result for connections is proved for nonsplit Lie 
algebroids. Finally, in Section \ref{sec5} some applications of the results are given when the Lie algebroid is 
related to the usual Atiyah exact sequence of a holomorphic vector bundle.

\section{Split Lie algebroids and connections}\label{sec2}

\subsection{Differential operators}

Let $X$ be a compact connected Riemann surface. The holomorphic tangent bundle
of $X$ will be denoted by $TX$, while the holomorphic cotangent bundle 
of $X$ will be denoted by $K_X$. The first jet bundle of a holomorphic vector bundle
$W$ over $X$ will be denoted by $J^1(W)$; so $J^1(W)$ is a holomorphic vector bundle with
${\rm rank}(J^1(W))\,=\, 2\cdot {\rm rank}(W)$ on $X$
that fits in the following short exact sequence of holomorphic vector bundles on $X$:
\begin{equation}\label{e0}
0\, \longrightarrow\, W\otimes K_X \, \longrightarrow\, J^1(W) \, \longrightarrow\, W
\, \longrightarrow\, 0.
\end{equation}

For a holomorphic vector bundle $F$ on $X$, the sheaf of holomorphic sections of the vector
bundle $${\rm Diff}^1(W,\, F)\ :=\ F\otimes J^1(W)^* \ =\ \text{Hom}(J^1(W),\, F)$$ is the
sheaf of holomorphic differential operators of order one from $W$ to $F$. Tensoring the dual of
\eqref{e0} with $F$ we get the short exact sequence
\begin{equation}\label{e-1}
0\, \longrightarrow\, F\otimes W^* \, \longrightarrow\, F\otimes J^1(W)^*\,=\,{\rm Diff}^1(W,\, F)
\end{equation}
$$
\stackrel{\sigma}{\longrightarrow}\, F\otimes (W\otimes K_X)^*\,=\, \text{Hom}(W,\, F)\otimes TX
\, \longrightarrow\, 0.
$$
The homomorphism $\sigma$ in \eqref{e-1} is called the \textit{symbol map}.

\subsection{Lie algebroid connections}

A $\mathbb{C}$--Lie algebra structure on a holomorphic vector
bundle $V$ on $X$ is a $\mathbb{C}$--bilinear pairing defined by a sheaf homomorphism
$$
[-,\, -] \,\,:\,\, V\otimes_{\mathbb C} V \,\, \longrightarrow\,\, V,
$$
which is given by a holomorphic homomorphism $J^1(V)\otimes J^1(V)\, \longrightarrow\, V$
of vector bundles, such that
$$[s,\, t]\,=\, -[t,\, s]\ \ \text{ and }\ \ [[s,\, t],\, u]+[[t,\, u],\, s]+[[u,\, s],\, t]\,=\,0$$
for all locally defined holomorphic sections $s,\, t,\, u$ of $V$. The Lie bracket operation
on $TX$ gives the structure of a Lie algebra on it. A Lie algebroid on $X$ is a
pair $(V,\, \phi)$, where
\begin{enumerate}
\item $V$ is a holomorphic vector bundle on $X$ equipped with the structure of a 
$\mathbb{C}$--Lie algebra, and

\item $\phi\, :\, V\, \longrightarrow\, TX$ is an ${\mathcal O}_X$--linear homomorphism such that
$$
[s,\, f\cdot t]\,=\, f\cdot [s,\, t]+\phi(s)(f)\cdot t
$$
for all locally defined holomorphic sections
$s,\, t$ of $V$ and all locally defined holomorphic functions $f$ on $X$.
\end{enumerate}
The above homomorphism $\phi$ is called the \textit{anchor map} of the Lie algebroid.

\begin{remark}
\label{rmk:anchor}
It can be shown (see, for instance \cite[Remark 2.1]{ABKS}) that for a Lie algebroid $(V,\, \phi)$ on $X$, we have
\begin{equation}\label{er}
\phi([s,\, t])\,=\, [\phi(s),\, \phi(t)]
\end{equation}
for all locally defined holomorphic sections $s,\, t$ of $V$.
\end{remark}

\begin{definition}\label{def1}
A Lie algebroid $(V,\, \phi)$ on $X$ will be called \textit{split} if there is
an ${\mathcal O}_X$--linear homomorphism
$$
\gamma\, :\, TX \, \longrightarrow\, V
$$
such that $\phi\circ\gamma\,=\, {\rm Id}_{TX}$. A Lie algebroid $(V,\, \phi)$ on $X$ will
be called \textit{nonsplit} if it is not split.
\end{definition}

Take a Lie algebroid $(V,\, \phi)$ on $X$. We have the dual homomorphism
\begin{equation}\label{e2}
\phi^*\,:\, K_X\, \longrightarrow\, V^*
\end{equation}
of the anchor map $\phi$. A \textit{Lie algebroid connection} on a holomorphic vector bundle
$E$ over $X$ is a first order holomorphic differential operator
$$
D\,\,:\,\, E\,\,\longrightarrow\, \, E\otimes V^*,
$$
meaning $D\, \in\, H^0(X, \, \text{Diff}^1(E,\, E\otimes V^*))$, such that the following
Leibniz type identity holds:
\begin{equation}\label{e-4}
D(fs) \,=\, fD(s) + s\otimes \phi^*(df)
\end{equation}
for all locally defined holomorphic sections $s$ of $E$ and all locally defined holomorphic
functions $f$ on $X$, where $\phi^*$ is the homomorphism in \eqref{e2}.

For any Lie algebroid $(V,\,\phi)$, and any $n\, \geq\, 0$, there is a $\mathbb C$--linear homomorphism
$$
d^n_V\ :\ \bigwedge\nolimits^n V^* \ \longrightarrow\ \bigwedge\nolimits^{n+1} V^*
$$
which is constructed as follows: $d_V^0(f) = \phi^*(df)$ for any locally defined holomorphic function $f$
on $X$. To construct $d^n_V$ for $n\, \geq\, 1$, take locally defined holomorphic sections $\omega \,\in\,
\bigwedge^n V^*$ and $v_1,\,\cdots,\, v_{n+1} \,\in\, V$; then
$$
d^n_V(\omega)(v_1,\, \cdots,\,v_{n+1})\ =\ \sum_{i=1}^n (-1)^{i+1}\phi(v_i)(\omega(v_1,\,\cdots,\,
\widehat{v}_i,\, \cdots,\, v_{n+1}))
$$
$$
+\sum_{1\le i<j\le n+1}(-1)^{i+j} \omega\left ([v_i,\, v_j],\, v_1,\,\cdots,\,
\widehat{v}_i,\, \cdots,\,\widehat{v}_j,\, \cdots,\, v_{n+1}\right).
$$
{}From \eqref{e-4} it follows that
$$\left (\bigwedge\nolimits^\bullet V^*,\,d_V \right) \ =\
\bigoplus_{n\geq 0} \left(\bigwedge\nolimits^n V^*,\,d^n_V\right)$$
is a differential graded complex; it
is called the Chevalley-Eilenberg-de Rham complex for $(V,\,\phi)$ (see \cite{BMRT}, \cite{LSX},
\cite{BR} for details).
Note that when $(V,\,\phi)\,=\, (TX,\,{\rm Id}_{TX})$, then $(\bigwedge\nolimits^\bullet V^*,\,d_V)$
is the holomorphic de Rham complex of $X$.

For a Lie algebroid connection $D\,:\, E\,\longrightarrow\, E\otimes V^*$, consider the
following composition of operators
$$
E\,\,\stackrel{D}{\longrightarrow}\,\, E\otimes V^* \,\,\xrightarrow{\,\,\,D\wedge {\rm Id}_{V^*}+
{\rm Id}_E\otimes d^1_V\,\,\,} \,\, E\otimes \bigwedge\nolimits^2 V^*.
$$
It is straightforward to check that this operator $E\, \longrightarrow\, E\otimes \bigwedge\nolimits^2 V^*$
is ${\mathcal O}_X$--linear, and hence it is given by a holomorphic section
$$
{\mathcal K}(D)\,\, \in\,\, H^0(X,\, \text{End}(E)\otimes \bigwedge\nolimits^2 V^*).
$$
This section ${\mathcal K}(D)$ is called the \textit{curvature} of $D$. The 
Lie algebroid connection $D$ is called \textit{integrable} (or \textit{flat}) if
we have ${\mathcal K}(D)\,=\, 0$.

Given a Lie algebroid connection $D\,:\,E\,\longrightarrow\, E\otimes V^*$, and a
holomorphic section $v\,\in\, H^0(U,\, V\big\vert_U)$ on an open subset $U\,\subset\, X$, let
\begin{equation}\label{dv}
D_v\,:\, E\big\vert_U\,\longrightarrow\, E\big\vert_U
\end{equation}
be the map that sends any $s\, \in\, H^0(U,\, E\big\vert_U)$ to the contraction
$\langle v,\, D(s)\rangle$ of $D(s)$ by $v$. Note that \eqref{e-4} is equivalent to the following:
\begin{equation}\label{dv2}
D_v (fs)\ =\ f\cdot D_v (s)+ \phi(v)(f)\cdot s
\end{equation}
for all $v\, \in\, H^0(U,\, V\big\vert_U)$, $f\, \in\, H^0(U,\, {\mathcal O}_U)$ and
$s\, \in\, H^0(U,\, E\big\vert_U)$ for every open subset $U\, \subset\, X$.
The Lie algebroid connection $D$ is flat if and only if for any pair of locally
defined sections $v,\,w \,\in\, V$,
$$D_{[v,\, w]}\ =\ [D_v,\,D_w]\ :=\ D_v\circ D_w - D_w\circ D_v .$$

Consider the very special Lie algebroid $(V,\, \phi)$, where $V\,=\, TX$ and
$\phi\,=\, {\rm Id}_{TX}$. Then a Lie algebroid connection on a holomorphic
vector bundle $E$ on $X$ is same as a holomorphic connection on $E$. (See \cite{At2}
for holomorphic connections.) Note that any holomorphic connection on $E$ is integrable
because $\bigwedge\nolimits^2 K_X\,=\, 0$.

\begin{lemma}\label{lem1}
Let $(V,\, \phi)$ be a Lie algebroid on $X$. Let $E$ be a holomorphic vector bundle
over $X$ equipped with an usual holomorphic connection $D\, :\, E\,\longrightarrow\, E\otimes K_X$.
Then $D$ induces a Lie algebroid connection on $E$ for the Lie algebroid
$(V,\, \phi)$. The induced Lie algebroid connection on $E$ is integrable.
\end{lemma}

\begin{proof}
The composition of homomorphisms
$$
E\,\,\stackrel{D}{\longrightarrow}\, E\otimes K_X\,\,\xrightarrow{\,\,\, {\rm Id}_E
\otimes \phi^*\,\,\,}\,\, E\otimes V^*,
$$
where $\phi^*$ is the homomorphism in \eqref{e2} is evidently a
Lie algebroid connection on $E$ for $(V,\, \phi)$. To verify that this composition, which will
be denoted by $D'$,
is flat, let $v$ and $w$ be two locally defined holomorphic sections of $V$. Then, by \eqref{er}, we have
$\phi([v,\,w])\,=\,[\phi(v),\,\phi(w)]$, so
$$D'_{[v,w]}\, =\, D_{\phi([v,w])}\, =\, D_{[\phi(v),\phi(w)]}.$$
On the other hand, as $D$ is flat,
$$D_{[\phi(v),\phi(w)]}\, =\, [D_{\phi(v)},\, D_{\phi(w)}]\, =\, [D'_v,\, D'_w].$$
So we have $D'_{[v,w]}\, =\, [D'_v,\, D'_w]$ and, thus, $D'$ is flat.
\end{proof}

\subsection{Split Lie algebroids}

A holomorphic vector bundle $W$ on $X$ is called \textit{decomposable} if there are holomorphic vector
bundles $W_1$ and $W_2$ --- both of positive ranks --- such that $W_1\oplus W_2$ is holomorphically
isomorphic to $W$. A holomorphic vector bundle $W$ on $X$ is called \textit{indecomposable} if it
is not decomposable. Every holomorphic vector bundle is isomorphic to a direct sum of indecomposable
vector bundles. If
\begin{equation}\label{e3}
\bigoplus_{i=1}^m U_i \,=\, W \,=\, \bigoplus_{j=1}^n V_i
\end{equation}
are two holomorphic decompositions of $W$ into direct sum of indecomposable vector bundles,
then a theorem of Atiyah says the following:
\begin{enumerate}
\item $m\,=\, n$, and

\item there is a permutation $\sigma$ of $\{1,\, \cdots,\, m\}$ such that $U_i$ is holomorphically
isomorphic to $V_{\sigma (i)}$ for all $1\, \leq\, i\, \leq\, m$. 
\end{enumerate}
(See \cite[p.~315, Theorem 2(ii)]{At1}.)

A holomorphic vector bundle $U$ is called an \textit{indecomposable component} of $W$ if
\begin{itemize}
\item $U$ is indecomposable, and

\item there is a holomorphic vector bundle $U'$ such that $W\,=\, U\oplus U'$.
\end{itemize}
Note that the above result of Atiyah (\cite[p.~315, Theorem 2(ii)]{At1}) implies that $W$ has
only finitely many indecomposable components.

A theorem of Weil and Atiyah says that a holomorphic vector bundle $W$ on $X$ admits a holomorphic connection
if and only if the degree of each indecomposable component of $W$ is zero \cite[p.~203, Theorem 10]{At2},
\cite{We}.

\begin{proposition}\label{prop1}
Let $(V,\, \phi)$ be a split Lie algebroid on $X$. Let $E$ be a holomorphic vector bundle on $X$.
Then the following three statements are equivalent:
\begin{enumerate}
\item Each indecomposable component of $E$ is of degree zero.

\item $E$ admits a Lie algebroid connection for $(V,\, \phi)$.

\item $E$ admits an integrable Lie algebroid connection for $(V,\, \phi)$.
\end{enumerate}
\end{proposition}

\begin{proof} ${\bf (1) \,\implies\, (3)}$: Suppose that (1) holds. Then \cite[p.~203, Theorem 10]{At2}
says that $E$ admits a holomorphic connection. Now (3) follows from Lemma \ref{lem1}.

It is obvious that (3) implies (2).

${\bf (2) \,\implies\, (1)}$: Suppose that (2) holds. Let
$$D\,\,:\,\, E\,\,\longrightarrow\, \, E\otimes V^*$$
be a Lie algebroid connection on $E$.

Since $(V,\, \phi)$ is a split Lie algebroid, there is a homomorphism
$$
\gamma\, :\, TX \, \longrightarrow\, V
$$
such that $\phi\circ\gamma\,=\, {\rm Id}_{TX}$. Let
$$
\gamma^*\, :\, V^* \, \longrightarrow\, K_X
$$
be the dual of $\gamma$. The following composition of maps
$$
E\,\,\stackrel{D}{\longrightarrow}\, \, E\otimes V^*\,\, \xrightarrow{\,\,\, {\rm Id}_E\otimes\gamma^*\,\,\,}
\,\, E\otimes K_X
$$
will be denoted by $D'$. It is easy to see that this $D'$ is a holomorphic connection on $E$. Indeed, for
each locally defined holomorphic section $s$ of $E$ and each local holomorphic function $f$ on $X$, we have
\begin{multline*}
D'(fs) \, =\, ({\rm Id}_E\otimes \gamma^*)(D(fs))\, =\, ({\rm Id}_E\otimes \gamma^*)(fD(s)+s\otimes \phi^*(df))\\
 =\, f\cdot ({\rm Id}_E\otimes \gamma^*)(D(s)) + s\otimes \gamma^*\phi^*(df) \, =\,
fD'(s)+s\otimes (\phi\circ \gamma)^*(df)\, =\, fD'(s)+s\otimes df
\end{multline*}
(recall that $\phi\circ\gamma\,=\, {\rm Id}_{TX}$). So $D'$ is a holomorphic connection on $E$. Consequently, the
Atiyah--Weil criterion says that the degree of each indecomposable component of $E$ is zero.
\end{proof}

In Section \ref{sec4} we will prove that for a nonsplit Lie algebroid on $X$, every holomorphic vector
bundle on $X$ admits a Lie algebroid connection.

\section{Connections and Atiyah exact sequence}\label{sec3}

For a Lie algebroid $(V,\, \phi)$ on $X$, consider the homomorphism
\begin{equation}\label{e5}
\phi\otimes {\rm Id}_{V^*}\, : \, \text{End}(V)\,=\,
V\otimes V^* \, \longrightarrow\, (TX)\otimes V^*.
\end{equation}
We have ${\mathcal O}_X \,=\, {\mathcal O}_X\cdot {\rm Id}_V \, \subset\, \text{End}(V)$. Let
\begin{equation}\label{e4}
\phi_0\, :\, {\mathcal O}_X \, \longrightarrow\, (TX)\otimes V^*
\end{equation}
be the restriction of the homomorphism $\phi\otimes {\rm Id}_{V^*}$ in \eqref{e5} to
this subbundle ${\mathcal O}_X \, \subset\, \text{End}(V)$.

Take a holomorphic vector bundle $E$ on $X$. Consider the symbol map
\begin{equation}\label{e6}
\sigma\, :\, \text{Diff}^1(E,\, E\otimes V^*)\, \longrightarrow\, \text{Hom}(E,\, E\otimes V^*)\otimes TX
\,=\, \text{End}(E)\otimes (TX)\otimes V^*
\end{equation}
(see \eqref{e-1}). Denoting the constant function $1$ on $X$ by $1_X$, we have
\begin{equation}\label{e7}
\phi_0(1_X)\, \in\, H^0(X,\, (TX)\otimes V^*),
\end{equation}
where $\phi_0$ is the homomorphism in \eqref{e4}. A direct computation shows that through the canonical isomorphism
$H^0(X,\, (TX)\otimes V^*) \, = \, H^0(X,\, \text{Hom}(V,TX))$, the section $\phi_0(1_X)$
in \eqref{e7} corresponds to the anchor map $\phi$.

Taking this into account, the following lemma is a straightforward consequence of the definition of a Lie algebroid connection
on $E$.

\begin{lemma}\label{lem0}
A Lie algebroid connection $D$ on $E$ is a holomorphic differential operator
$$
D\,\, \in\,\, H^0(X,\, {\rm Diff}^1(E,\, E\otimes V^*))
$$
such that $\sigma(D)\,=\, {\rm Id}_E\otimes \phi_0(1_X) \, \in\, H^0(X,\, {\rm End}(E)\otimes (TX)\otimes V^*)$,
where $\sigma$ and $\phi_0(1_X)$ are constructed in \eqref{e6} and \eqref{e7} respectively.
\end{lemma}

Consider the short exact sequence
\begin{equation}\label{e8}
0\,\longrightarrow\, \text{End}(E)\otimes V^*\, \longrightarrow\,
\text{Diff}^1(E,\, E\otimes V^*)\, \stackrel{\sigma}{\longrightarrow}\,
\text{End}(E)\otimes (TX)\otimes V^*\, \longrightarrow\, 0
\end{equation}
(see \eqref{e6} and \eqref{e-1}). We also have the homomorphism
\begin{equation}\label{e10}
{\rm Id}_E\otimes \phi_0 \, :\, {\mathcal O}_X \,=\, {\rm Id}_E\otimes {\mathcal O}_X\,
\longrightarrow\, \text{End}(E)\otimes (TX)\otimes V^*,
\end{equation}
where $\phi_0$ is the homomorphism in \eqref{e4}. Using \eqref{e10}, construct the subsheaf
\begin{equation}\label{e12}
\text{Diff}^1_1(E,\, E\otimes V^*)
\ :=\ \sigma^{-1}({\rm Id}_E\otimes\phi_0 ({\mathcal O}_X))\
\subset\ \text{Diff}^1(E,\, E\otimes V^*),
\end{equation}
where $\sigma$ is the surjective homomorphism in \eqref{e8}. The restriction of $\sigma$ to
this subbundle $\text{Diff}^1_1(E,\, E\otimes V^*)\, \subset\, \text{Diff}^1(E,\, E\otimes V^*)$ will
be denoted by $\widehat\sigma$.

Lemma \ref{lem0} has the following consequence:

\begin{corollary}\label{cor1}
A Lie algebroid connection $D$ on $E$ is a holomorphic section 
$$
D'\,\, \in\,\, H^0(X,\, {\rm Diff}^1_1(E,\, E\otimes V^*)),
$$
where ${\rm Diff}^1_1(E,\, E\otimes V^*)$ is constructed in \eqref{e12}, such that the section
$$\widehat{\sigma}(D')\, \in\, H^0(X,\, {\rm End}(E)\otimes (TX)\otimes V^*),$$
where $\widehat\sigma$ is the restriction of $\sigma$ (constructed in \eqref{e8}), coincides with ${\rm Id}_E
\otimes \phi_0 (1_X)$ (see \eqref{e10}), where $1_X$, as before, is the constant function $1$ on $X$.
\end{corollary}

\begin{proof}
If $D\, \in\, H^0(X,\, {\rm Diff}^1(E,\, E\otimes V^*))$ gives a Lie algebroid connection on $E$ (see
Lemma \ref{lem0}), then we clearly have
$$
D\, \in\, H^0(X,\, \text{Diff}^1_1(E,\, E\otimes V^*))\,\subset\,
H^0(X,\, {\rm Diff}^1(E,\, E\otimes V^*)).
$$
The condition that $\widehat{\sigma}(D)\,=\, {\rm Id}_E\otimes \phi_0 (1_X)$ is satisfied by Lemma \ref{lem0}.

Conversely, any $D'\, \in\, H^0(X,\, \text{Diff}^1_1(E,\, E\otimes V^*))$, with $\widehat{\sigma}(D')\,=\,
{\rm Id}_E\otimes \phi_0 (1_X)$, gives a Lie algebroid connection on $E$ by Lemma \ref{lem0}.
\end{proof}

Note that the homomorphism ${\rm Id}_E\otimes \phi_0$ in \eqref{e10} is nonzero if and only if the homomorphism of
coherent analytic sheaves given by ${\rm Id}_E\otimes \phi_0$ is injective. Clearly,
the homomorphism ${\rm Id}_E\otimes \phi_0$ is nonzero if and only if we have $\phi\, \not=\, 0$.

\begin{lemma}\label{lem2}
Let $(V,\, \phi)$ be a Lie algebroid on $X$ with $\phi\,=\, 0$. Then any holomorphic vector bundle $E$
on $X$ admits an integrable Lie algebroid connection for $(V,\, \phi)$.
\end{lemma}

\begin{proof}
Since $\phi\,=\, 0$, a Lie algebroid connection on $E$ for $(V,\, \phi)$ is same as a holomorphic section
$$\theta\, \,\in\,\, H^0(X,\, \text{End}(E)\otimes V^*).$$
In particular, $\theta \,=\, 0$ defines a Lie algebroid connection on $E$ for $(V,\, \phi)$. For $\theta \,=\, 0$, we have
$${\mathcal K}(\theta) \ =\ 0,$$
and hence the corresponding Lie algebroid connection is integrable.
\end{proof}

Let $(V,\, \phi)$ be a Lie algebroid on $X$ with $\phi\,\not=\, 0$. So the homomorphism of coherent analytic sheaves
${\rm Id}_E\otimes \phi_0$ in \eqref{e10} is injective. We will identify ${\mathcal O}_X$ with its image ${\rm Id}_E
\otimes\phi_0 ({\mathcal O}_X)\, \subset\, \text{End}(E)\otimes (TX)\otimes V^*$ using ${\rm Id}_E\otimes \phi_0$.
Consequently, from \eqref{e8} and \eqref{e12} we have the exact sequence
\begin{equation}\label{e9}
0\,\longrightarrow\, \text{End}(E)\otimes V^*\, \longrightarrow\,
\text{Diff}^1_1(E,\, E\otimes V^*)\, \stackrel{\widehat{\sigma}}{\longrightarrow}\,
{\mathcal O}_X \, \longrightarrow\, 0
\end{equation}
where $\widehat{\sigma}$ as before is the restriction of $\sigma$ in \eqref{e8}.
Tensoring \eqref{e9} with $V$ we get the exact sequence
\begin{equation}\label{e11}
0\ \longrightarrow\ \text{End}(E)\otimes V^*\otimes V\ =\ \text{End}(E)\otimes \text{End}(V)
\end{equation}
$$
\stackrel{\iota}{\longrightarrow}\ \text{Diff}^1_1(E,\, E\otimes V^*)\otimes V \
\xrightarrow{\,\,\,\,\widehat{\sigma}\otimes {\rm Id}_V\,\,\,\,}\
V \ \longrightarrow\ 0.
$$
Let $\text{End}^0(V)\, \subset\, \text{End}(V)$ be the sheaf of endomorphisms of $V$ of trace zero. So we have
\begin{equation}\label{e16}
\text{End}(V)\, =\, \text{End}^0(V)\oplus {\mathcal O}_X\cdot {\rm Id}_V\,=\, \text{End}^0(V)\oplus {\mathcal O}_X. 
\end{equation}
Quotienting \eqref{e11} by $\text{End}(E)\otimes \text{End}^0(V)$ we have the following exact sequence:
\begin{equation}\label{e13}
0\ \longrightarrow\ (\text{End}(E)\otimes \text{End}(V))/(\text{End}(E)\otimes \text{End}^0(V))
\ = \ \text{End}(E)
\end{equation}
$$
\longrightarrow\, {\mathcal A}_E\ :=\
(\text{Diff}^1_1(E,\, E\otimes V^*)\otimes V)/\iota(\text{End}(E)\otimes \text{End}^0(V))
\ \xrightarrow{\,\,\,\rho\,\,\,}\ V \ \longrightarrow\ 0,
$$
where $\iota$ is the homomorphism in \eqref{e11} and $\rho$ is given by $\widehat{\sigma}\otimes {\rm Id}_V$
in \eqref{e11}; note that \eqref{e16} is used in the above identification of $(\text{End}(E)\otimes
\text{End}(V))/(\text{End}(E)\otimes \text{End}^0(V))$ with $\text{End}(E)$.

Consider the following exact sequence (see \eqref{e-1}):
\begin{equation}\label{e14}
0\, \longrightarrow\, \text{End}(E) \, \longrightarrow\, {\rm Diff}^1(E,\, E)
\, \stackrel{\sigma_0}{\longrightarrow}\, \text{End}(E)\otimes TX
\, \longrightarrow\, 0.
\end{equation}
We recall that the Atiyah bundle ${\rm At}(E)$ for $E$ is
$$
{\rm At}(E)\ :=\ \sigma^{-1}_0({\rm Id}_E\otimes TX)\ \subset\ {\rm Diff}^1(E,\, E)
$$
(see \cite{At2}). So \eqref{e14} gives the exact sequence
\begin{equation}\label{e15}
0\, \longrightarrow\, \text{End}(E) \, \longrightarrow\, {\rm At}(E)
\, \stackrel{\sigma_0}{\longrightarrow}\, {\rm Id}_E\otimes TX\,=\, TX
\, \longrightarrow\, 0,
\end{equation}
which is known as the Atiyah exact sequence for $E$. A holomorphic connection $E$ is a holomorphic homomorphism
$h\, :\, TX\, \longrightarrow\, {\rm At}(E)$ such that $\sigma_0\circ h\, =\, {\rm Id}_{TX}$, where $\sigma_0$
is the homomorphism in \eqref{e15} (see \cite{At2}); in fact this is a reformulation of Corollary
\ref{cor1} for $(V,\, \phi)\,=\, (TX,\, {\rm Id}_{TX})$.

The vector bundle ${\rm At}(E)$ has a natural Lie algebroid structure, where the Lie bracket operation is
the commutator of differential operators, and the anchor map is $\sigma_0$ in \eqref{e15}.

\begin{proposition}\label{prop2}
There is a natural homomorphism $\varphi\, :\,{\mathcal A}_E \, \longrightarrow\, {\rm At}(E)$ 
such that the following diagram is commutative:
$$
\begin{matrix}
0 &\longrightarrow & {\rm End}(E) & \longrightarrow & {\mathcal A}_E &
\xrightarrow{\,\,\,\rho\,\,\,} & V & \longrightarrow & 0\\
&& \Big\Vert &&\,\,\, \Big\downarrow\varphi &&\,\,\, \Big\downarrow \phi\\
0 & \longrightarrow & {\rm End}(E) & \longrightarrow & {\rm At}(E)
& \stackrel{\sigma_0}{\longrightarrow} & TX & \longrightarrow & 0
\end{matrix}
$$
(see the exact sequences in \eqref{e13} and \eqref{e15}).
\end{proposition}

\begin{proof}
Consider the subsheaf
$$
\text{Diff}^1_1(E,\, E\otimes V^*)\otimes V\, \subset\, \text{Diff}^1(E,\, E\otimes V^*)\otimes V
\, =\, \text{Diff}^1(E,\, E)\otimes V^*\otimes V
$$
$$
\,=\, \text{Diff}^1(E,\, E)\otimes \text{End}(V)
$$
(see \eqref{e12} and \eqref{e-1}). The trace map $\text{tr}\, : \, \text{End}(V)\,\longrightarrow\, {\mathcal O}_X$, which coincides with the projection constructed
using the decomposition in \eqref{e16}, produces a projection
\begin{equation}\label{e17}
\text{Diff}^1_1(E,\, E\otimes V^*)\otimes V\, \hookrightarrow\, \text{Diff}^1(E,\, E)\otimes \text{End}(V)
\, \stackrel{p}{\longrightarrow}\, \text{Diff}^1(E,\, E)\otimes {\mathcal O}_X\,=\, \text{Diff}^1(E,\, E).
\end{equation}
The homomorphism $\sigma$ in \eqref{e8} produces a homomorphism
$$
\widetilde{\sigma}\, \,:\,\, \text{Diff}^1(E,\, E)\otimes \text{End}(V)\,\,=\,\,
\text{Diff}^1(E,\, E\otimes V^*)\otimes V
$$
$$
\xrightarrow{\,\,\, \sigma\otimes {\rm Id}_V
\,\,\,}\,\, \text{End}(E)\otimes (TX)\otimes V^*\otimes V\,\,=\,\, \text{End}(E)\otimes (TX)\otimes \text{End}(V).
$$
The projection $p={\rm Id}_{\text{Diff}^1(E,\, E)}\otimes \text{tr}$ in \eqref{e17} fits in the following commutative diagram:
\begin{equation}\label{e18}
\begin{matrix}
\text{Diff}^1(E,\, E)\otimes \text{End}(V) & \stackrel{p}{\longrightarrow} &\text{Diff}^1(E,\, E)\\
\,\,\, \,\Big\downarrow \widetilde{\sigma} &&\,\,\,\,\Big\downarrow\sigma_0\\
\text{End}(E)\otimes (TX)\otimes \text{End}(V) & \longrightarrow & \text{End}(E)\otimes (TX)
\end{matrix}
\end{equation}
where $\sigma_0$ is the projection in \eqref{e14} and the above
homomorphism
$$\text{End}(E)\otimes (TX)\otimes \text{End}(V) \,
\longrightarrow\, \text{End}(E)\otimes (TX)$$
is ${\rm Id}_{\text{End}(E)\otimes TX}$ tensored
with the trace map $\text{tr}\, : \, \text{End}(V)\,\longrightarrow\, {\mathcal O}_X$. From \eqref{e18} it follows immediately that
\begin{equation}\label{e18b}
p(\text{Diff}^1_1(E,\, E\otimes V^*)\otimes V)\ \subset \ {\rm At}(E)\ \subset\ \text{Diff}^1(E,\, E).
\end{equation}

Next observe that $(p\circ\iota) (\text{End}(E)\otimes \text{End}^0(V))\,=\, 0$, where
$\iota$ is the homomorphism in \eqref{e11}. Consequently, from \eqref{e18b} and the definition of
${\mathcal A}_E$ in \eqref{e13} we conclude that $p$ produces an homomorphism
$$
\varphi\, :\,{\mathcal A}_E \, \longrightarrow\, {\rm At}(E)
$$ 
{}From the construction of $\varphi$ it follows immediately that the diagram in the proposition is actually commutative.
\end{proof}

\begin{remark}\label{rem2}
When $V\,=\, TX$ and $\phi\,=\, {\rm Id}_{TX}$, then from the commutative diagram in Proposition
\ref{prop2} it follows immediately that $\varphi$ is an isomorphism, because in that case the other
two vertical maps in the diagram are isomorphisms. This can also be seen directly.
\end{remark}

\begin{lemma}\label{lem3}
As before, $(V,\, \phi)$ is a Lie algebroid with $\phi\, \not=\, 0$. Giving a Lie algebroid connection
on a holomorphic vector bundle $E$ on $X$ for $(V,\, \phi)$ is equivalent to giving a homomorphism
$$
\delta\ :\ V\ \longrightarrow\ {\mathcal A}_E
$$
such that $\rho\circ\delta \,=\, {\rm Id}_V$, where $\rho$ is the homomorphism in \eqref{e13}.
\end{lemma}

\begin{proof}
Take a Lie algebroid connection on $E$. It gives a section $D'\, \in\, H^0(X,\, {\rm Diff}^1_1(E,\, E\otimes V^*))$
such that $\widehat{\sigma}(D')\, =\, {\rm Id}_E\otimes \phi_0 (1_X)$ (see Corollary \ref{cor1}). Let
$$
D''\, :\, {\mathcal O}_X \, \longrightarrow\, {\rm Diff}^1_1(E,\, E\otimes V^*)
$$
be the homomorphism defined by $f\, \longmapsto\, f\cdot D'$. Consider the homomorphism
$$
D''\otimes {\rm Id}_V\, :\, {\mathcal O}_X\otimes V\,=\, V
\, \longrightarrow\, {\rm Diff}^1_1(E,\, E\otimes V^*)\otimes V.
$$
Denote by $\delta$ the composition of homomorphisms
\begin{equation}\label{f1}
V\ \xrightarrow{\,\,\, D''\otimes {\rm Id}_V\,\,\,}\ {\rm Diff}^1_1(E,\, E\otimes V^*)\otimes V
\end{equation}
$$
\xrightarrow{\,\,\,\, q\,\,\,\,}\ (\text{Diff}^1_1(E,\, E\otimes V^*)\otimes V)/\iota(\text{End}(E)\otimes
\text{End}^0(V))\ =\ {\mathcal A}_E,
$$
where $q$ is the natural quotient map (see \eqref{e13}). Then using the given condition that
$\widehat{\sigma}(D')\, =\, {\rm Id}_E\otimes \phi_0 (1_X)$ it follows that $\rho\circ\delta \,=\, {\rm Id}_V$,
where $\rho$ is the homomorphism in \eqref{e13} and $\delta$ is the composition of maps in \eqref{f1}.

To prove the converse, take a homomorphism
$$
\delta\ :\ V\ \longrightarrow\ {\mathcal A}_E
$$
such that
\begin{equation}\label{el}
\rho\circ\delta \ =\ {\rm Id}_V.
\end{equation}
Consider the map $D'\, = \, \varphi \circ \delta \, : \, V \longrightarrow {\rm At}(E) \subset {\rm Diff}^1(E,E)$,
where $\varphi$ is constructed in Proposition \ref{prop2}. From \eqref{el} it follows that
$$
D'(v) (fs)\ =\ f\cdot D'(v) (s)+ \phi(v)(f)\cdot s
$$
for all $v\, \in\, H^0(U,\, V\big\vert_U)$, $f\, \in\, H^0(U,\, {\mathcal O}_U)$ and
$s\, \in\, H^0(U,\, E\big\vert_U)$ for every open subset $U\, \subset\, X$ (see \eqref{dv2}).
Consequently, the differential operator $$D\ :\ E\ \longrightarrow\ E\otimes V^*$$
defined by $\langle v,\, D(s)\rangle \,=\, D'(v)(s)$, where $\langle v,\, -\rangle$ denotes the
contraction by $v$, satisfies \eqref{e-4}. Thus $D$ is a Lie algebroid connection on $E$ for $(V,\, \phi)$.
\end{proof}

\begin{remark}
In his thesis \cite{To1, To2}, Tortella described a notion of a generalized Atiyah bundle for Lie algebroids
through a different approach. He showed that when the sheaf $\mathcal{J}_V^1(E)\, :=\, V^* \otimes E \oplus E$
is given the following $\mathcal{O}_X$-module structure
$$f\cdot (\alpha\otimes s \oplus s')\, := \, ( f\alpha\otimes s + \phi^*(df)\otimes s) \oplus (fs')$$
for each locally defined holomorphic function $f$ of $X$, each pair of local sections $s$ and $s'$ of $E$
and each local section $\alpha\in V^*$, then there exists a short exact sequence
$$0\,\longrightarrow\, V^*\otimes E \,\longrightarrow\,\mathcal{J}_V^1(E)\,\longrightarrow\,E\,\longrightarrow\, 0$$
such that its splittings are also in correspondence with Lie algebroid connections on $E$.
\end{remark}

Let
\begin{equation}\label{elc}
D\,:\, E\,\longrightarrow\, E\otimes V
\end{equation}
be a Lie algebroid connection on $E$ for $(V,\, \phi)$. Consider the homomorphism
$$
\delta\ :\ V\ \longrightarrow\ {\mathcal A}_E
$$
corresponding to $D$ given by Lemma \ref{lem3}. Take holomorphic sections $v$ and $w$ of
$V$ defined on an open subset $U\, \subset\, X$. We have a holomorphic vector field
$$
[\sigma_0\circ \varphi\circ\delta(v),\, \sigma_0\circ \varphi\circ\delta(w)]-
\sigma_0\circ \varphi\circ\delta ([v,\, w])\ \in\ H^0(U,\, TU)
$$
(see the diagram in Proposition \ref{prop2} for $\sigma_0$ and $\varphi$). From the commutativity of the
diagram in Proposition \ref{prop2} we know that $$\sigma_0\circ \varphi\circ\delta\,=\, 
\phi\circ\rho\circ \delta\,=\, \phi ,$$
and hence from \eqref{er} it follows that
$$[\sigma_0\circ \varphi\circ\delta(v),\, \sigma_0\circ \varphi\circ\delta(w)]- \sigma_0\circ
\varphi\circ\delta ([v,\, w])\,=\, [\phi(v),\, \phi(w)] - \phi ([v,\, w])\, =\, 0.$$
Therefore, by \eqref{er}, we have
$$\sigma_0\left([\varphi\circ\delta(v),\, \varphi\circ\delta(w)]- 
\varphi\circ\delta ([v,\, w]) \right)\, =\, [\sigma_0\circ \varphi\circ\delta(v),\, \sigma_0\circ \varphi\circ\delta(w)]- \sigma_0\circ
\varphi\circ\delta ([v,\, w])\, =\, 0.$$
Thus, from
\eqref{e15} we have
$$
[\varphi\circ\delta(v),\, \varphi\circ\delta(w)] - \varphi\circ\delta ([v,\, w])\ \in\ H^0(U,\, \ker(\sigma_0)) \, = \, H^0(U,\, \text{End}(E)).
$$
It is straightforward to check that $$[\varphi\circ\delta(f_1v),\, \varphi\circ\delta(f_2w)]
- \varphi\circ\delta( [f_1v,\, f_2w])
\,=\, f_1f_2 \left ([\varphi\circ\delta(v),\, \varphi\circ\delta(w)]- \varphi\circ\delta([v,\, w]) \right)$$
for any holomorphic functions $f_1,\, f_2$ on $U$. Also, we have
$$
[\varphi\circ\delta(w),\, \varphi\circ\delta(v)]- \varphi\circ\delta([w,\, v])
\, = \, -([\varphi\circ\delta(v),\, \varphi\circ\delta(w)]- \varphi\circ\delta([v,\, w])). 
$$
These together imply that the above map
$$
(v,\, w)\ \longmapsto\ [\varphi\circ\delta(v),\, \varphi\circ\delta(w)] - \varphi\circ\delta([v,\, w])
$$
produces a holomorphic section
\begin{equation}\label{th}
\Theta\ \in\ H^0(X,\, \text{End}(E)\otimes \bigwedge\nolimits^2 V^*).
\end{equation}

The following lemma is straightforward.

\begin{lemma}\label{ll}
The curvature ${\mathcal K}(D)$ of the Lie algebroid connection $D$ in \eqref{elc} is the section $\Theta$ constructed in \eqref{th}.
\end{lemma}

\section{Nonsplit Lie algebroids and connections}\label{sec4}

Fix a Lie algebroid $(V,\, \phi)$ on $X$ such that $\phi\, \not=\, 0$. Take a holomorphic vector bundle $E$ on $X$.
Note that
$$
\text{End}(E)^*\,=\, (E\otimes E^*)^* \,=\, E^*\otimes E \,=\, E\otimes E^*
\,=\, \text{End}(E);
$$
this isomorphism is given by trace pairing $A\otimes B\, \longmapsto\, \text{trace}(AB)$ on $\text{End}(E)$.
Consider the Atiyah exact sequence for $E$ (see \eqref{e15}). Let
\begin{equation}\label{e21b}
\lambda\ \in\ H^1(X,\, \text{End}(E)\otimes K_X)
\end{equation}
be the extension class for the Atiyah exact sequence for $E$. Using Serre duality, and the fact that $\text{End}(E)^*\,
=\, \text{End}(E)$, we have
\begin{equation}\label{e21}
\lambda\ \in\ H^1(X,\, \text{End}(E)\otimes K_X) \ =\ H^0(X,\, \text{End}(E))^*\ = \ \text{Hom}(H^0(X,\, \text{End}(E)),
\, {\mathbb C}).
\end{equation}

Next consider the exact sequence in \eqref{e13}. Let
\begin{equation}\label{e22}
\lambda_\phi\ \in\ H^1(X,\, \text{End}(E)\otimes V^*) \ = \ H^0(X,\, \text{End}(E)\otimes V\otimes K_X)^*
\end{equation}
$$
=\ \text{Hom}(H^0(X,\, \text{End}(E)\otimes V\otimes K_X),\, {\mathbb C})
$$
be the extension class for it; the first isomorphism is given by
Serre duality. We will describe its relationship with $\lambda$ in \eqref{e21}.

We have the homomorphism
$$
\beta\, :=\, \text{Id}_{\text{End}(E)}\otimes \phi^*\, :\, \text{End}(E)\otimes K_X\, \longrightarrow\,
\text{End}(E)\otimes V^*,
$$
where $\phi^*$ is the dual of the anchor map (see \eqref{e2}). Let
\begin{equation*}%\label{e23}
\beta_*\ :\ H^1(X,\, \text{End}(E)\otimes K_X)\ \longrightarrow\ H^1(X,\, \text{End}(E)\otimes V^*)
\end{equation*}
be the homomorphism of cohomologies induced by $\beta$. From the commutative diagram in Proposition \ref{prop2}
it follows immediately that
\begin{equation}\label{e24}
\beta_* (\lambda) \ =\ \lambda_\phi,
\end{equation}
where $\lambda$ and $\lambda_\phi$ are the cohomology classes in \eqref{e21b} and \eqref{e22} respectively.

Tensoring $\phi\, :\, V\, \longrightarrow\, TX$ with ${\rm Id}_{K_X}$ we have
$$
\phi\otimes {\rm Id}_{K_X}\, :\, V\otimes K_X \, \longrightarrow\, (TX)\otimes K_X \,=\, {\mathcal O}_X.
$$
Tensoring $\phi\otimes {\rm Id}_{K_X}$ with $\text{Id}_{\text{End}(E)}$ we have
\begin{equation}\label{e25a}
\beta'\, :=\, \text{Id}_{\text{End}(E)}\otimes (\phi\otimes {\rm Id}_{K_X})\, :\, \text{End}(E)\otimes V\otimes K_X\, \longrightarrow\,
\text{End}(E).
\end{equation}
Let
\begin{equation}\label{e25}
\beta'_*\ :\ H^0(X,\, \text{End}(E)\otimes V\otimes K_X)\ \longrightarrow\ H^0(X,\, \text{End}(E))
\end{equation}
be the homomorphism of global sections induced by $\beta'$ in \eqref{e25a}.

{}From \eqref{e24} it follows immediately that the following diagram is commutative:
\begin{equation}\label{e26}
\begin{matrix}
H^0(X,\, \text{End}(E)\otimes V\otimes K_X) & \xrightarrow{\,\,\,\,\lambda_\phi\,\,\,} & {\mathbb C}\\
\,\,\,\, \Big\downarrow \beta'_* && \Big\Vert\\
H^0(X,\, \text{End}(E)) & \xrightarrow{\,\,\,\,\lambda\,\,\,} & {\mathbb C}
\end{matrix}
\end{equation}
where $\lambda$, $\lambda_\phi$ and $\beta'_*$ are constructed in \eqref{e21}, \eqref{e22} and \eqref{e25} respectively.
 
\begin{proposition}\label{prop3}
Assume that the following three statements hold:
\begin{enumerate}
\item $\phi\, \not=\, 0$,

\item the Lie algebroid $(V,\, \phi)$ is nonsplit (see Definition \ref{def1}), and

\item the holomorphic vector bundle $E$ on $X$ is indecomposable.
\end{enumerate}
Then $E$ admits a Lie algebroid connection for $(V,\, \phi)$.
\end{proposition}

\begin{proof}
In view of Lemma \ref{lem3}, the vector bundle $E$ admits a Lie algebroid connection if and only
if $\lambda_\phi\,=\, 0$ (see \eqref{e22} for $\lambda_\phi$). So to prove the proposition it suffices to show that
\begin{equation}\label{e27}
\lambda_\phi\ =\ 0.
\end{equation}

Atiyah proved in \cite{At2} that $\lambda\,:\,H^0(X,\, \text{End}(E))\, \longrightarrow\, {\mathbb C}$
in \eqref{e21} is the zero homomorphism if $\lambda (\text{Id}_E)\,=\, 0$. We briefly recall his proof.

Take any $A\, \in\, H^0(X,\, \text{End}(E))$. The coefficients of the characteristic polynomial of A are
constants because there are no nonconstant holomorphic functions on $X$. So the eigenvalues of $A$ are
constants over $X$. Note that $A$ has only one eigenvalue because $E$ is indecomposable, and multiple eigenvalues
would give a nontrivial eigenbundle decomposition of $E$. So
$$
A\ = \ \mu\cdot {\rm Id}_E + N,
$$
where $N\,\in\, H^0(X,\, \text{End}(E))$ is a nilpotent endomorphism of $E$ and $\mu$ is the unique eigenvalue of $A$.
Assume that $N\, \not=\, 0$. Using $N$ we get a holomorphic reduction of structure group
of $E$ to the group of lower-triangular invertible matrices (a Borel subgroup) of $\text{GL}(r,{\mathbb C})$, where
$r\,=\,\text{rank}(E)$. The endomorphism $N$ is strictly lower-triangular
with respect to this reduction of structure group. Since the trace pairing of a lower-triangular
$r\times r$ matrix and a strictly lower-triangular $r\times r$ matrix is zero, we have $\lambda (N)\,=\, 0$. Hence
we conclude that $\lambda\,=\, 0$ if $\lambda ({\rm Id}_E)\,=\, 0$.

To prove \eqref{e27} by contradiction, assume that $\lambda_\phi\, \not=\, 0$. So from \eqref{e26} we conclude that
there is a section
$$
B\ \in\ H^0(X,\, \text{End}(E)\otimes V\otimes K_X)
$$
such that
\begin{equation}\label{e28}
\lambda\circ \beta'_* (B)\ \not= \ 0.
\end{equation}
Consider $\beta'_*(B)\, \in\, H^0(X,\, \text{End}(E))$. It was observed above (recalled from \cite{At2}) that
$\beta'_*(B)$ has exactly one eigenvalue, because $E$ is indecomposable. This eigenvalue is nonzero, because
otherwise $\beta'_*(B)$ would be nilpotent, and $\lambda$ evaluated on a nilpotent endomorphism of $E$ is zero
(this was also observed above and recalled from \cite{At2}), which would contradict \eqref{e28}. Since
$\beta'_*(B)$ has exactly one eigenvalue, and the eigenvalue is nonzero, it follows that
\begin{equation}\label{e29}
\text{trace}(\beta'_* (B))\ \not= \ 0.
\end{equation}

For any homomorphism
$\varpi\, :\, TX\, \longrightarrow\, V$, consider the composition of maps
$$
TX\, \xrightarrow{\,\,\,\, \varpi\,\,\,\,} \, V \, \xrightarrow{\,\,\,\, \phi\,\,\,\,} \, TX.
$$
It coincides with $c\cdot \text{Id}_{TX}$ for some $c\, \in\, {\mathbb C}$. Let
\begin{equation}\label{e31}
\Phi\ :\ H^0(X,\, V\otimes K_X) \ \longrightarrow\ {\mathbb C}
\end{equation}
be the homomorphism that sends any $\varpi\, :\, TX\, \longrightarrow\, V$
to the above complex number $c\,\in\, {\mathbb C}$ given by $\phi\circ\varpi$.

For $B$ in \eqref{e28}, we have
\begin{equation}\label{e30}
{\mathcal B}\ :=\ {\rm trace}(B) \ \in\ H^0(X,\, V\otimes K_X),
\end{equation}
where ``trace'' denotes the trace map $\text{End}(E)\, \longrightarrow\, {\mathcal O}_X$ tensored with
${\rm Id}_{V\otimes K_X}$. From the construction of the maps $\beta'$ in \eqref{e25a} and $\Phi$ in \eqref{e31} we conclude that
\begin{equation}\label{e32}
\Phi({\mathcal B})\ =\ \text{trace}(\beta'_* (B)),
\end{equation}
where $\mathcal B$ and $\beta'_*$ are constructed in \eqref{e30} and \eqref{e25} respectively. Now
\eqref{e29} and \eqref{e32} together imply that
\begin{equation}\label{e33}
\Phi({\mathcal B})\ \not=\ 0.
\end{equation}
Recall from the construction of $\Phi$ that the composition of maps
$$
TX\, \xrightarrow{\,\,\,\, {\mathcal B}\,\,\,\,} \, V \, \xrightarrow{\,\,\,\, \phi\,\,\,\,} \, TX
$$
coincides with multiplication by $\Phi({\mathcal B})\,\in\, {\mathbb C}$. So from \eqref{e33} it follows that the homomorphism
$$
\frac{1}{\Phi({\mathcal B})}{\mathcal B}\ :\ TX \ \longrightarrow\ V
$$
has the property that
\begin{equation}\label{e34}
\phi\circ \left(\frac{1}{\Phi({\mathcal B})}{\mathcal B}\right)\,=\, {\rm Id}_{TX}.
\end{equation}

But \eqref{e34} contradicts the given condition that the Lie algebroid $(V,\, \phi)$ is nonsplit. Hence
\eqref{e27} is proved. This completes the proof of the proposition.
\end{proof}

\begin{corollary}\label{cor2}
Let $(E,\,\phi)$ be a Lie algebroid on $X$. Assume that the following two statements hold:
\begin{enumerate}
\item $\phi\, \not=\, 0$, and

\item $(V,\, \phi)$ is nonsplit.
\end{enumerate}
Then any holomorphic vector bundle $E$ on $X$ admits a Lie algebroid connection for $(V,\, \phi)$.
\end{corollary}

\begin{proof}
Express $E$ as a direct sum of indecomposable vector bundles
$$
E\ = \ \bigoplus_{i=1}^\ell E_i.
$$
Each $E_i$, $1\,\leq\, i\, \leq\, \ell$, admits a Lie algebroid connection for $(V,\, \phi)$
by Proposition \ref{prop3}. If $D_i$ is a Lie algebroid connection on $E_i$, then
$$
D\ = \ \bigoplus_{i=1}^\ell D_i
$$
is a Lie algebroid connection on $E$ for $(V,\, \phi)$.
\end{proof}

\begin{theorem}\label{thm1}
Let $(E,\,\phi)$ be a nonsplit Lie algebroid on $X$. Then any holomorphic vector bundle $E$ on $X$ admits a
Lie algebroid connection for $(V,\, \phi)$.
\end{theorem}

\begin{proof}
If $\phi\, =\, 0$, then $E$ admits an integrable Lie algebroid connection by Lemma \ref{lem2}.

If $\phi\, \not=\, 0$, then $E$ admits a Lie algebroid connection by Corollary \ref{cor2}.
\end{proof}

\section{Example of Lie algebroid connections}\label{sec5}

Let $F$ be a holomorphic vector bundle on $X$. Set the underlying vector bundle $V$ for the Lie
algebroid to be the Atiyah bundle ${\rm At}(F)$ (see \eqref{e15}), and set the anchor map
$\phi\, :\, {\rm At}(F)\, \longrightarrow\, TX$ to be the natural projection as in \eqref{e15}.
Then the Lie algebroid is split if and only if $F$ admits an usual holomorphic connection \cite{At2}, \cite{We}.

Proposition \ref{prop1} gives the following:

\begin{corollary}\label{cor3}
Assume that $F$ admits an usual holomorphic connection, and consider the above Lie algebroid
$({\rm At}(F),\,\phi)$. Let $E$ be a holomorphic vector bundle on $X$.
Then the following three statements are equivalent:
\begin{enumerate}
\item Each indecomposable component of $E$ is of degree zero.

\item $E$ admits a Lie algebroid connection for $({\rm At}(F),\,\phi)$.

\item $E$ admits an integrable Lie algebroid connection for $({\rm At}(F),\,\phi)$.
\end{enumerate}
\end{corollary}

\begin{proof}
Since $F$ admits an usual holomorphic connection, the Lie algebroid $({\rm At}(F),\,\phi)$ is split.
Hence Proposition \ref{prop1} shows that the above three statements are equivalent.
\end{proof}

A special case of Corollary \ref{cor3}: Set $F$ to be an indecomposable holomorphic vector bundle
of degree zero.

Theorem \ref{thm1} gives the following:

\begin{corollary}\label{cor4}
Assume that $F$ does not admit any usual holomorphic connection, and consider the above Lie algebroid
$({\rm At}(F),\,\phi)$ associated to $F$. Then any holomorphic vector bundle $E$ on $X$ admits a
Lie algebroid connection for $({\rm At}(F),\,\phi)$.
\end{corollary}

A special case of Corollary \ref{cor4}: Set $F$ to be a holomorphic vector bundle
of nonzero degree.

Let $E$ be a holomorphic vector bundle on $X$. Consider its Atiyah exact sequence as in
\eqref{e15}:
\begin{equation}\label{e41}
0\, \longrightarrow\, \text{End}(E) \, \longrightarrow\, {\rm At}(E)
\, \stackrel{\phi}{\longrightarrow}\, TX \, \longrightarrow\, 0.
\end{equation}
We have the Lie algebroid $(\text{At}(E),\, \phi)$. Then it can be shown that $E$ has a tautological Lie
algebroid connection for the Lie algebroid $(\text{At}(E),\, \phi)$; furthermore, the tautological Lie
algebroid connection on $E$ is integrable. To see the tautological Lie algebroid connection on $E$, 
first note that the vector bundle ${\mathcal A}_E$ in \eqref{e13} is the subbundle of $\text{At}(E)
\oplus\text{At}(E)$ consisting of all $(v,\, w)\, \in\, \text{At}(E)\oplus \text{At}(E)$ such that
$\phi(v)\,=\, \phi(w)$, where $\phi$ is the projection in \eqref{e41}. So the exact sequence in
\eqref{e13} becomes
\begin{equation}\label{e42}
0\, \longrightarrow\, \text{End}(E) \, \longrightarrow\, {\mathcal A}_E
\, \stackrel{p_2}{\longrightarrow}\, \text{At}(E) \, \longrightarrow\, 0,
\end{equation}
where $p_2$ is the restriction to ${\mathcal A}_E\, \subset\, \text{At}(E)\oplus \text{At}(E)$
of the natural projection $\text{At}(E)\oplus \text{At}(E)\, \longrightarrow\,
\text{At}(E)$ to the second factor. The projection $p_2$ in \eqref{e42} has a natural holomorphic
splitting $\text{At}(E)\, \longrightarrow\, {\mathcal A}_E$ defined by $v\, \longmapsto\, (v,\, v)$.
This splitting defines a Lie algebroid connection on $E$ (for $(\text{At}(E),\, \phi)$) by Lemma
\ref{lem3}. From Lemma \ref{ll} it follows that this Lie algebroid connection on $E$ is integrable.

Let $E$ be an indecomposable holomorphic vector bundle on $X$.
As before, let $$\text{End}^0(E)\, \subset\, \text{End}(E)$$ be the subbundle of endomorphisms of
trace zero. Quotienting \eqref{e41} by $\text{End}^0(E)$ we get the exact sequence
\begin{equation}\label{e43}
0\, \longrightarrow\, \text{End}(E)/\text{End}^0(E)\,=\, {\mathcal O}_X
\, \longrightarrow\, V_E\,:=\, {\rm At}(E)/\text{End}^0(E)
\, \stackrel{\widehat{\phi}}{\longrightarrow}\, TX \, \longrightarrow\, 0,
\end{equation}
where $\widehat{\phi}$ is induced by $\phi$ in \eqref{e41}. So $(V_E,\, \widehat{\phi})$ in
\eqref{e43} is a Lie algebroid on $X$. Now consider the exact sequence
in \eqref{e13} for $E$ (for the Lie algebroid $(V_E,\, \widehat{\phi})$):
\begin{equation}\label{e43a}
0\, \longrightarrow\, \text{End}(E) \,\longrightarrow\, {\mathcal A}_E
\, \xrightarrow{\,\,\,\rho\,\,\,}\, V_E \, \longrightarrow\, 0.
\end{equation}
{}From the commutative diagram in Proposition \ref{prop2} we conclude the following:
The vector bundle ${\mathcal A}_E$ in \eqref{e43a} is the subbundle of
$\text{At}(E)\oplus V_E$ consisting of all $(v,\, w)\, \in\, \text{At}(E)\oplus V_E$ such
that $\phi(v)\,=\, \widehat{\phi}(w)$, where $\phi$ and $\widehat\phi$ are the projections in
\eqref{e41} and \eqref{e43} respectively; the homomorphism $\rho$ in \eqref{e43a} is the
restriction, to ${\mathcal A}_E$, of the natural projection $\text{At}(E)\oplus V_E\,
\longrightarrow\, V_E$.

We also have ${\mathcal O}_X\, \subset\, \text{End}(E)$ (see \eqref{e16}).
Quotienting \eqref{e41} by ${\mathcal O}_X$ we get the exact sequence
\begin{equation}\label{e44}
0\, \longrightarrow\, \text{End}(E)/{\mathcal O}_X \,=\, \text{End}^0(E) 
\, \longrightarrow\, \text{At}^0(E)\,:=\, {\rm At}(E)/{\mathcal O}_X
\, \stackrel{\phi'}{\longrightarrow}\, TX \, \longrightarrow\, 0,
\end{equation}
where $\phi'$ is given by $\phi$ in \eqref{e41}. We recall that a holomorphic projective connection
on $E$ is given by a holomorphic homomorphism $$\delta\, :\, TX\, \longrightarrow\, \text{At}^0(E)$$
such that $\phi'\circ\delta \,=\, \text{Id}_{TX}$, where $\phi'$ is the projection in \eqref{e44}.

Since the vector bundle $E$ is indecomposable, it admits
a holomorphic projective connection \cite{AB}. Let
\begin{equation}\label{e45}
\delta\, :\, TX\, \longrightarrow\, \text{At}^0(E)
\end{equation}
be a homomorphism giving a holomorphic projective connection on $E$. Let
$$
q\, :\, \text{At}(E)\, \longrightarrow\, \text{At}^0(E)\,:=\, {\rm At}(E)/{\mathcal O}_X
$$
be the quotient map. Consider the inverse image
$$
F_\delta \ :=\ q^{-1}(\delta(TX))\ \subset \ \text{At}(E),
$$
where $\delta$ is the homomorphism in \eqref{e45}. Note that the following composition of maps
$$
F_\delta \, \hookrightarrow\, \text{At}(E)\, \longrightarrow\, 
{\rm At}(E)/\text{End}^0(E) \, =: \, V_E
$$
is an isomorphism (the vector bundle $V_E$ is constructed in \eqref{e43}); let
\begin{equation}\label{e46}
\Psi\, :\, F_\delta \, \longrightarrow\, V_E
\end{equation}
be the isomorphism given by this composition of maps.

Consider the homomorphism
\begin{equation}\label{eP}
\Phi\ :\ V_E \ \longrightarrow\ \text{At}(E)\oplus V_E
\end{equation}
defined by $v\, \longmapsto\, (\Psi^{-1}(v),\, v)$, where $\Psi$ is the isomorphism in
\eqref{e46} (note that $F_\delta \, \subset\, \text{At}(E)$, so $\Psi^{-1}(v)\, \in\,
\text{At}(E)$). The image of $\Phi$ in \eqref{eP} evidently lies in the subbundle 
${\mathcal A}_E\, \subset\, \text{At}(E)\oplus V_E$, where ${\mathcal A}_E$ is the vector bundle
in \eqref{e43a} (it was shown earlier that ${\mathcal A}_E$ is
a subbundle of $\text{At}(E)\oplus V_E$). Now it is easy to see that the map
$\Phi\, :\, V_E \, \longrightarrow\, {\mathcal A}_E$ satisfies the condition
$$
\rho\circ \Phi\ =\ \text{Id}_{V_E},
$$
where $\rho$ is the projection in \eqref{e43a}.
In view of Lemma \ref{lem3}, this implies that $\Phi$ gives a Lie algebroid connection $E$
for the Lie algebroid $(V_E,\, \widehat{\phi})$ in \eqref{e43}.

Since the holomorphic projective connection on $E$ given by $\delta$ in \eqref{e45} is integrable,
it follows that the above Lie algebroid connection on $E$
(for the Lie algebroid $(V_E,\, \widehat{\phi})$) is also integrable.

\section*{Acknowledgements}

D.A. was supported by grants PID2022-142024NB-I00 and RED2022-134463-T funded by 
MCIN/AEI/10.13039/501100011033. A.S. is partially supported by SERB SRG Grant SRG/2023/001006.
I.B. is partially supported by a J. C. Bose Fellowship (JBR/2023/000003).

\section*{Data availability}

No datasets were generated, used or analyzed in this paper.

\end{document}